\documentclass[11pt]{article}

\usepackage{amsmath,amssymb,amsfonts,amsthm,amsbsy,bm}
\usepackage{fullpage}
\usepackage[usenames]{color}
\usepackage{ulem}
\usepackage{epsfig}
\usepackage{bm}
\usepackage{natbib}
\newcommand{\be}{\begin{equation}}
\newcommand{\ee}{\end{equation}}

\newtheorem{proposition}{Proposition}[section]

\newtheorem{lemma}[proposition]{Lemma}
\newtheorem{definition}[proposition]{Definition}

\newtheorem{theorem}[proposition]{Theorem}
\newtheorem{fact}[proposition]{Fact}

\usepackage{authblk}
\title{Compressed sensing and optimal denoising of monotone signals}
\author{Eftychios A. Pnevmatikakis}
\affil{Center for Computational Biology, Flatiron Institute, Simons Foundation, New York, NY 10010}
\date{}
\begin{document}
%
\maketitle
\begin{abstract}
We consider the problems of compressed sensing and optimal denoising for signals $\bm{x_0}\in\mathbb{R}^N$ that are monotone, i.e., $\bm{x_0}(i+1) \geq \bm{x_0}(i)$, and sparsely varying,  i.e., $\bm{x_0}(i+1) > \bm{x_0}(i)$ only for a small number $k$ of indices $i$. We approach the compressed sensing problem by minimizing the total variation norm restricted to the class of monotone signals subject to equality constraints obtained from a number of measurements $A\bm{x_0}$. For random Gaussian sensing matrices $A\in\mathbb{R}^{m\times N}$ we derive a closed form expression for the number of measurements $m$ required for successful reconstruction with high probability. We show that the probability undergoes a phase transition as $m$ varies, and depends not only on the number of change points, but also on their location. For denoising we regularize with the same norm and derive a formula for the optimal regularizer weight that depends only mildly on $\bm{x_0}$. We obtain our results using the statistical dimension tool. 
\end{abstract}
%
%
\section{Introduction}
We consider $N$-dimensional signals $\bm{x_0}$ that are sparse in a certain basis. We are interested in the following problems
\begin{flalign}
	&&\min_{\bm{x}}f(\bm{x}), \text{ subject to } A\bm{x} = A\bm{x_0}, && \label{PCS} \tag{CS} \\
	\text{and} &&\min_{\bm{x}}\dfrac{1}{2}\|\bm{y}-\bm{x}\|^2 + \lambda f(\bm{x}), &&\label{PDN} \tag{DN}
\end{flalign}
with $\bm{y} = \bm{x_0} + \bm{\varepsilon}$ and $\bm{\varepsilon} \sim \mathcal{N}(0,\sigma^2I_N)$. Here $f:\mathbb{R}^N\mapsto \mathbb{R}\cup\{\infty\}$ is a convex function that characterizes the structure of $\bm{x_0}$. For the compressed sensing problem \eqref{PCS} we are interested in deriving the minimum number of measurements $m$ such that the solution of \eqref{PCS} coincides with $\bm{x_0}$ with high probability for standard normal i.i.d. sensing matrices $A\in\mathbb{R}^{m\times N}$. For the denoising problem \eqref{PDN} we are interested in calculating the minimax risk, i.e., the optimal value of  \eqref{PDN} minimized over $\lambda$ for the worst case of noise power $\sigma^2$. The two quantities are closely related due to some recent results reviewed briefly next.

\section{Basic tools}

\begin{definition}[Descent cones]
	The descent cone of a convex function $f:\mathbb{R}^N\mapsto \mathbb{R}$ at a point $\bm{x} \in \mathbb{R}^N$ is defined as the set of all non-increasing directions, i.e.,
	\[ \mathcal{D}(f,\bm{x}) = \bigcup_{\tau>0}\{\bm{y}\in\mathbb{R}^N: f(\bm{x}+\tau\bm{y}) \leq f(\bm{x})\}. \]
	\label{DC-def}
\end{definition}
\begin{definition}[Statistical dimension \citep{ALMT14}]
	The statistical dimension (SD) of a convex closed cone $\mathcal{C} \in \mathbb{R}^N$ is defined as 
		\[ \delta(\mathcal{C}) = \mathbb{E}_{\bm{g} \sim \mathcal{N}(\bm{0},I_N)}\|\Pi_{\mathcal{C}}(\bm{g})\|^2, \]
	where $\bm{g}$ is a standard Gaussian vector, and $\Pi_{\mathcal{C}}$ is the projection onto $\mathcal{C}$.
\end{definition}
\noindent In a groundbreaking work, \cite{ALMT14} shows that the SD of the descent cone at the true point $\bm{x_0}$, coincides with the phase transition curve (PTC) of the \ref{PCS} problem.
\begin{theorem}[Phase transitions \citep{ALMT14}]
	For an i.i.d. standard random Gaussian matrix $A \in \mathbb{R}^{m\times N}$ the convex problem \eqref{PCS} succeeds with probability at least $1 - \exp(-t^2/4)$ if 
	\[ m \geq \delta(\mathcal{D}(f,\bm{x_0})) + t\sqrt{N}, \]
	and fails with probability at least $1 - \exp(-t^2/4)$ if
	\[ m \leq \delta(\mathcal{D}(f,\bm{x_0})) - t\sqrt{N}. \]
	\label{PTC-thm}
\end{theorem}
\vskip-0.5cm
\noindent Furthermore, \cite{ALMT14} shows that the SD can also be expressed as the expected distance from the subdifferential of $f$ at $\bm{x_0}$:
\be
	\delta(\mathcal{D}(f,\bm{x_0})) = \mathbb{E}_{\bm{g} \sim \mathcal{N}(\bm{0},I_N)}[\min_{\tau\geq 0}\mathrm{dist}(\bm{g},\tau\partial f(\bm{x_0}))^2]
	\label{sd_sub}
\ee
\begin{theorem}[Minimax risk \citep{oymak2012relation}]
	Let $\bm{x^*}(\lambda)$ the solution of the denoising problem \eqref{PDN} with regularizer weight $\lambda$ and let 
	\[ \eta_f(\bm{x_0}) = \min_{\lambda\geq 0}\max_{\sigma>0}\dfrac{\mathbb{E}\|\bm{x^*}(\lambda)-\bm{x_0}\|^2}{\sigma^2}, \]
	the minimax risk for $\bm{x_0}$ over all possible $\sigma$. Then: 
	\be \eta_f(\bm{x_0}) = \min_{\tau\geq 0}\mathbb{E}_{\bm{g}\sim\mathcal{N}(\bm{0},I)}[\mathrm{dist}(\bm{g},\tau\partial f(\bm{x_0}))^2], \label{minmax} \ee 
	where $\bm{g}$ is a standard normal vector. Moreover the risk is maximized for $\sigma\rightarrow 0$ and if $\tau^*$ is the value that minimizes \eqref{minmax}, then $\lambda^* = \tau^*\sigma$ is the optimal choice as $\sigma\rightarrow 0$.
	\label{mm_risk}
\end{theorem}
\noindent The similarity between \eqref{sd_sub} and \eqref{minmax} is striking and actually \cite{ALMT14} proves that the two quantities are indeed close:
\[
	\delta(\mathcal{D}(f,\bm{x_0})) \leq \eta_f(\bm{x_0}) \leq \delta(\mathcal{D}(f,\bm{x_0})) + 2\dfrac{\underset{\tiny{\bm{w}\in\partial f(\bm{x_0})}}{\sup}\|\bm{w}\|}{f(\bm{x_0}/\|\bm{x_0}\|)}.
\]

\section{Phase transitions for the recovery of sparsely varying monotone signals}

We consider signals $\bm{x}\in\mathbb{R}^N$ that are increasing, i.e., $\bm{x}(i+1) \geq \bm{x}(i)$ and are sparsely varying, i.e, $\bm{x}(i+1) > \bm{x}(i)$ for a number of $k$ indexes. 
A convex function that promotes this structure can be derived by restricting the total variation (TV) norm to the space of monotone signals:
\be
	f(\bm{x}) = \left\{ \begin{array}{ll} \bm{x}(N) - \bm{x}(1), & \bm{x}(i+1) \geq \bm{x}(i), i \in [N-1] \\ \infty, & \text{otherwise.} \end{array} \right.
	\label{NRTV}
\ee
where $[N] = \{1,2,\ldots,N\}$. 
Our results rely heavily on the following calculation of the SDs of the cones induced by monotone signals, proven in \citet[App.~C.4]{ALMT14}.
\begin{fact}
	Let the cones 
	\[ \begin{split}
		\mathcal{C}_1^N & = \{ \bm{x} \in \mathbb{R}^N: \bm{x}(1) \leq \bm{x}(2) \leq \ldots \leq \bm{x}(N) \} \\ 
		\mathcal{C}_2^N & = \{ \bm{x} \in \mathbb{R}^N: 0 \leq \bm{x}(1) \leq \bm{x}(2) \leq \ldots \leq \bm{x}(N)  \}.
	\end{split} \]
	Then we  have $\delta(\mathcal{C}_1^N) = H_N$, and $\delta(\mathcal{C}_2^N) = \frac{1}{2}H_N$, where $H_N = \sum_{i=1}^N\frac{1}{i}$, denotes the $N$-th harmonic number.
	\label{chambers}
\end{fact}

\subsection{Computation of the statistical dimension}
According to Theorem \ref{PTC-thm} to compute the PTC for the \ref{PCS} problem, we need to characterize $\mathcal{D}(f,\bm{x_0})$. 
\begin{lemma}
	Let $\Omega = \{ i \in \{2,\ldots,N\} : \bm{x_0}(i) > \bm{x_0}(i-1) \}$ and define $i_1 < i_2 < \ldots < i_k$ the elements of $\Omega$ in increasing order. The descent cone of the norm $f$ of \eqref{NRTV} at $\bm{x_0}$ is given by
	\be 
		\begin{split}
		&\mathcal{D}(f,\bm{x_0})  = \left\{\bm{y}\in\mathbb{R}^N: \right. \\
		&				 \left.
						\begin{array}{c}   \bm{y}(i_1) \leq \bm{y}(i_1+1) \leq \ldots \leq \bm{y}(i_2 - 1) \\
									\vdots \\
									\bm{y}(i_{k-1}) \leq \bm{y}(i_{k-1} + 1) \leq \ldots \leq \bm{y}(i_k -1) \\
									  \bm{y}(i_k) \leq \ldots \leq \bm{y}(N) \leq \bm{y}(1)  \leq \ldots \leq \bm{y}(i_1-1) 
						\end{array}\right\}.		\end{split}
		\label{desc_con_neg}
	\ee
	\label{desc_mon_neg}
\end{lemma}
\begin{proof}
	From Definition \ref{DC-def}, $\bm{y} \in \mathcal{D}(f,\bm{x_0})$ if there exists $\tau > 0$, such that $\bm{x_0} + \tau\bm{y}$ is monotone, and $f(\bm{x_0} + \tau\bm{y}) \leq f(\bm{x_0})$:
	\[ \begin{split}
		\bm{x_0}(N) + \tau\bm{y}(N) - \bm{x_0}(1) -  \tau\bm{y}(1) &\leq \bm{x_0}(N) - \bm{x_0}(1) \Rightarrow \\ \bm{y}(N) &\leq \bm{y}(1).
		\label{dec-cond}
	\end{split}
	\]
	For the monotonicity of $\bm{x_0} + \tau\bm{y}$, we consider two cases: If $i \not\in \Omega$, then $\bm{x_0}(i) = \bm{x_0}(i-1)$, and 
			\[
				\bm{x_0}(i) + \tau\bm{y}(i) \geq \bm{x_0}(i-1)+ \tau\bm{y}(i-1) \Rightarrow \bm{y}(i) \geq \bm{y}(i-1).
			\]
	If $i \in \Omega$, then $\bm{x_0}(i) > \bm{x_0}(i-1)$ and $\bm{y}(i)$ can be chosen arbitrarily since there is always a small enough $\tau$ that will preserve monotonicity.
	Combining everything we get \eqref{desc_con_neg}. 
\end{proof}

\noindent Lemma \ref{desc_mon_neg} states that the descent cone $\mathcal{D}(f,\bm{x_0})$ can be expressed as the product of $k$ disjoint convex cones of monotonically increasing signals. Using Fact \ref{chambers}, we derive the following simple formula for $\delta(\mathcal{D}(f,\bm{x_0}))$ as the sum of the SDs of the simpler disjoint cones.
\begin{theorem}
Let $\Omega = \{ i \in \{2,\ldots,N\} : \bm{x_0}(i) > \bm{x_0}(i-1) \}$ and define $i_1 < i_2 < \ldots < i_k$ the elements of $\Omega$ in increasing order. 
The SD of the descent cone at $\bm{x_0}$ equals
\be
	\delta(\mathcal{D}(f,\bm{x_0})) =  \sum_{j=2}^{k}H_{i_j-i_{j-1}} + H_{N+i_1-i_k}.
\ee
\end{theorem}

\subsection{Dependence on the change points location} 
\noindent The closed form of the SD allows for a characterization of the worst case analysis for a given number of variations $k$. These locations have to occur periodically every $N/k$ steps, with $i_1 \approx N/2k$. If $r_{N,k} = \mathrm{mod}(N,k)$, then the SD becomes $(k-r_{N,k})H_{[N/k]} + r_{N,k}H_{[N/k]+1}.$, where $[\cdot]$ here denotes the integer part. For moderately large $N/k$ this converges to $kH_{N/k} \rightarrow k(\log(N/k) + \gamma)$, where $\gamma \approx 0.577$ is the Euler-Mascheroni constant. 

Similarly, the best case occurs when all change points occur consecutively. In this case the SD becomes $(k-1) + H_{N+1-k}$. What is perhaps of most interest is the average SD under certain distribution assumptions of the $k$ change points. We can  asymptotically compute this in the case where these $k$ points are distributed uniformly at random.
\begin{theorem}
	Assume that the $k$ change points are chosen uniformly at random and let $N,k\rightarrow\infty$ with $k/N = \varepsilon$, $0< \varepsilon < 1$. Define $\delta_{U}(\varepsilon)$ the normalized (divided by the ambient dimension) SD averaged over all possible choices of $k=\varepsilon N$ ``jump" points. Then we have
	\be
		\delta_{U}(\varepsilon) = \dfrac{\varepsilon\log(1/\varepsilon)}{1-\varepsilon}.
		\label{ptc-avg}
	\ee
\end{theorem}

\proof
	Let $i_1<i_2<\ldots<i_k$ the change points selected uniformly randomly and define the sequence of lengths $l_j = i_{j+1} - i_j$ for $j\in[k-1]$ and $l_k = N - i_k + i_1$. 
	When $N,k\rightarrow\infty$ the distribution of each $l_j$ converges to a geometric distribution with parameter $\varepsilon = k/N$.
	Then we have 
	\[ \begin{split}
		\delta_{U}(\varepsilon)  &= \lim_{N\rightarrow\infty}\dfrac{1}{N}\mathbb{E}\left[\sum_{j=1}^kH_{l_j}\right] = \varepsilon\mathbb{E}\left[H_{l_j}\right]\\
						  & = \varepsilon^2 \sum_{n=1}^{\infty}H_n(1-\varepsilon)^{n-1} 
						   = \dfrac{\varepsilon^2}{1-\varepsilon}  \sum_{n=1}^{\infty} \dfrac{1}{n}\sum_{m=n}^{\infty}(1-\varepsilon)^{m} \\
						  & = \dfrac{\varepsilon^2}{1-\varepsilon} \sum_{n=1}^{\infty} \dfrac{1}{n} \dfrac{(1-\varepsilon)^{n}}{\varepsilon} =  \dfrac{\varepsilon\log(1/\varepsilon)}{1-\varepsilon}.
	\qed					 
	\end{split} 
	\]  

\begin{figure}
	\centering
	\includegraphics[trim={0mm 0mm 0mm 9.5mm}, clip, width=0.5\textwidth]{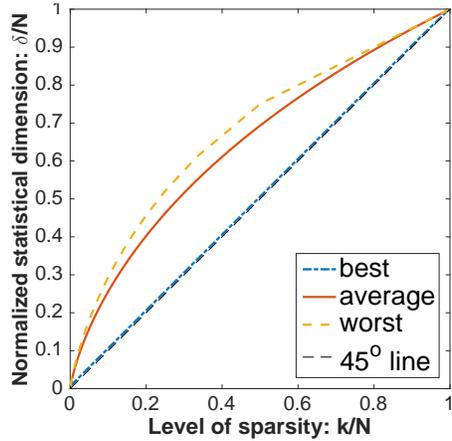}
	\caption{Behavior of the SD as a function of the degree sparsity and location of change points. The best (blue, dash-dot), average (red, solid), and worst (yellow, dashed) cases are shown.}
	\label{PTC-cases}
\end{figure}
\noindent Fig.~\ref{PTC-cases} shows the three different cases for the SD, and illustrates its dependence on the location of the jump points.
\begin{figure}
	\centering
	\includegraphics[trim={33mm 0mm 24mm 1.5mm}, clip, width=.75\textwidth]{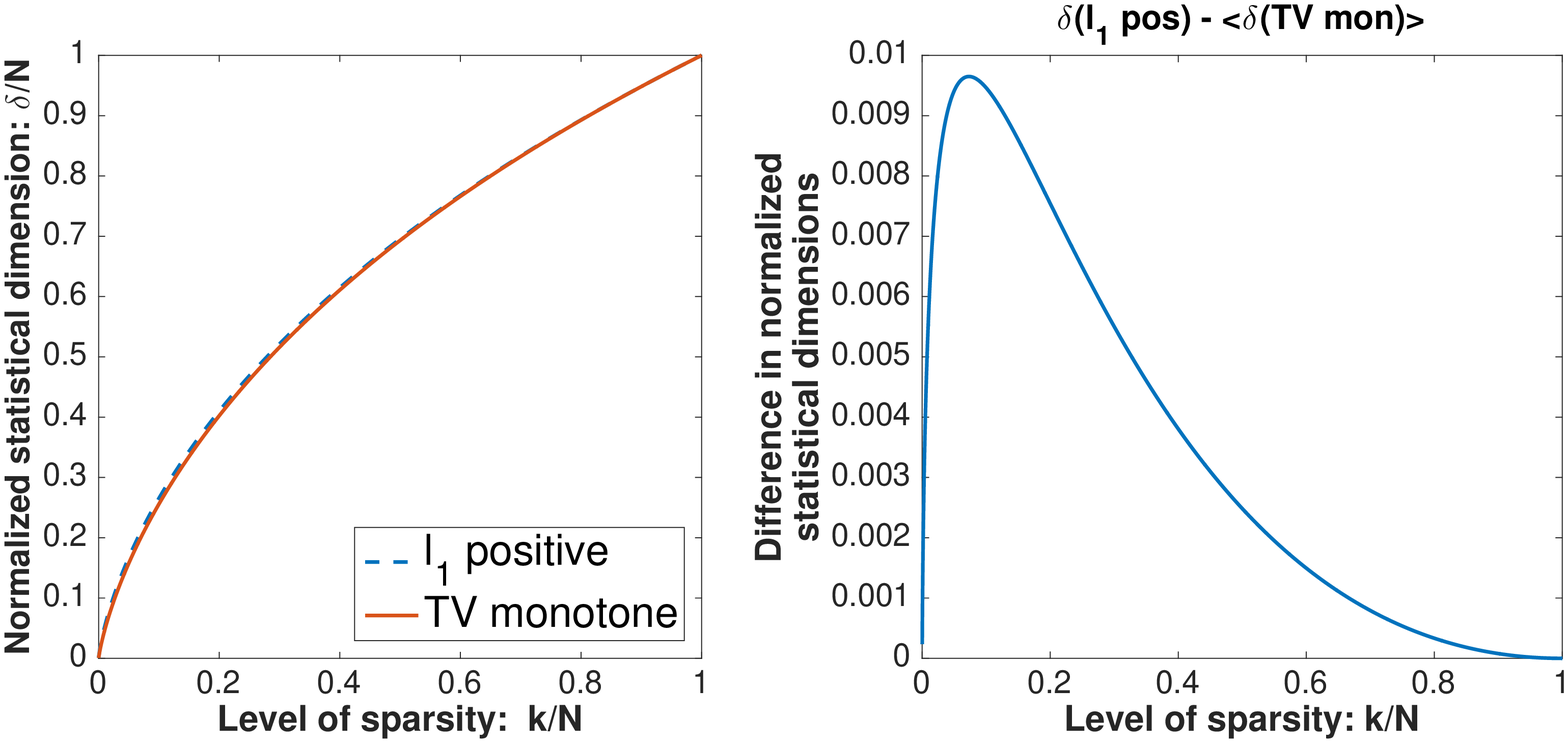}
	\caption{Comparison of the SD for monotone sparsely varying signals and sparse non-negative signals. The SD for reconstructing sparse non-negative signals (red, solid) as computed in \cite{DMM09} compared to the asymptotic average limit of \eqref{ptc-avg} (blue, dashed). Somewhat surprisingly, the two curves almost coincide with the positive $l_1$ norm PTC being always larger than the PTC of for the monotone signals for the same number of variations. The difference is plot in the right panel.}
	\label{PTC-plots}
\end{figure}

 In Fig.~\ref{PTC-plots} we plot the SD as computed in \eqref{ptc-avg} (blue), compared to the PTC for the reconstruction of sparse non-negative signals (red) as this is computed in \cite{DMM09}, using the $l_1$ norm restricted to non-negative signals as the structure inducing function $f$ and solving \eqref{PCS}. The two curves are very near, although the PTC curve for sparse non-negative signals is slightly larger. The difference between the two different curves (Fig.~\ref{PTC-plots} right) attains a maximum of $\approx0.0096$ for $k/N \approx 0.0731$. 
 We examined this difference in practice: 
 For the sparse non-negative signal we considered signals $\bm{x_0}\in\mathbb{R}^N, N = 1000$, with $k=73$ non-zero entries. Then random Gaussian matrices $A\in\mathbb{R}^{m\times N}$ were constructed, with $m=201,\ldots,220$, and we tried to reconstruct $\bm{x_0}$ from the samples $A\bm{x_0}$ by solving \eqref{PCS}. We solved the same problem also for the case of sparsely varying increasing signals, where now $k=73$, refers to the number of change points, chosen uniformly at random. For each $m$ we performed 300 iterations, and the reconstruction $\hat{\bm{x}}$ was deemed successful if $\|\hat{\bm{x}} - \bm{x_0}\|/\|\bm{x_0}\| \leq 10^{-4}$. The results show that the probability of accurate reconstruction crosses $50\%$ within one measurement from the point predicted by the theoretical calculation of the SD, and that the difference of measurements required for $50\%$ reconstruction probability is around 10 measurements, as predicted by the difference of the two SDs (data not shown due to space constraints). These simulation results validate the theoretical analysis.

\subsection{The case of non-negative monotone signals}
We also consider the case of increasing and sparsely varying signals $\bm{x} \in\mathbb{R}^N$, that are also non-negative, which we denote without loss of generality as $\bm{x}(0)=0$, and consider the first entry as a change point if $\bm{x}(1)>0$. In this case we consider the following convex regularizer $f(\bm{x}) = \bm{x}(N)$ for monotonically increasing signals and $f(\bm{x}) =\infty$ otherwise.
%
Using a similar procedure we can derive the descent cone $\mathcal{D}(f,\bm{x_0})$ and from Fact \ref{chambers} get a similar formula for the SD:
\begin{theorem}
Let $\Omega = \{ i \in [N] : \bm{x_0}(i) > \bm{x_0}(i-1) \}$ and define $i_1 < i_2 < \ldots < i_k$ the elements of $\Omega$ in increasing order. 
Then the SD of the descent cone at $\bm{x_0}$ is given by
\[
	\delta(\mathcal{D}(f,\bm{x_0})) = \dfrac{1}{2}H_{i_1-1} + \dfrac{1}{2}H_{N+1-i_k} + \sum_{j=2}^{k}H_{i_j-i_{j-1}}.
\]
\end{theorem}

\section{Optimal denoising}
For the case of monotone, sparsely varying, non-negative signals it is also possible to compute the minimax denoising risk, by using Theorem \ref{mm_risk}.
To consider the risk of the denoising problem \eqref{PDN}, we first derive the subdifferential of $f$. Let $G$ be the $N\times N$ matrix with $[G]_{ij} = 1_{\{i=j\}} - 1_{\{i=j+1\}}$ 
and define the function $h:\mathbb{R}^N\mapsto \mathbb{R}\cup\{\infty\}$, with 
\[
	h(\bm{z}) = \left\{ \begin{array}{ll} \bm{1}^{\top}\bm{z}, & \bm{z}(i) \geq 0, i \in [N] \\ \infty, & \text{otherwise.} \end{array} \right.
\]
Then $f(\bm{x}) = h(G\bm{x})$ and $\partial f(\bm{x}) = G^{\top}\partial h(G\bm{x})$ and 
\[
	\partial f(\bm{x}) = G^{\top}\bm{w}, \text{ with } \left\{ \begin{array}{cc} \bm{w}(i) = 1, & \bm{x}(i) > \bm{x}(i-1) \\  \bm{w}(i) \leq 1, & \bm{x}(i) = \bm{x}(i-1) \end{array} \right..
\]
Therefore the distance of any vector $\bm{g}\in\mathbb{R}^N$ from $\partial f(\bm{x_0})$ can be computed by solving the following quadratic program
\be
	\tag{QP}
	\begin{aligned}
		 & \underset{\bm{w},\tau}{\text{minimize }}  \|\bm{g}-G^{\top}\bm{w}\|^2, \\
		& \text{subject to:} \; \tau\geq0, \{w(j)=\tau, j\in \Omega\}, \{w(j) \leq \tau, j\in \Omega^c\}.
	\end{aligned}
	\label{quad}
\ee

\begin{lemma}
	Consider the quadratic program \eqref{quad} and let $i_k$ denote the last element of $\Omega$. Then the optimal $\tau^*$ is given by
	\be
		\tau^{\ast} = \max\left(\max_{j=i_k,\ldots,N}\left\{ \sum_{n=j}^{N}\bm{g}(n)\right\},0\right).
	\ee
	\label{opt_tau}
\end{lemma}
\proof 
	We consider the Lagrangian function 
\be
	\mathcal{L}(\bm{w},\tau,\bm{\lambda},\lambda_{\tau}) = \frac{1}{2}\|\bm{g} - G^{\top}\bm{w}\|^2 - \lambda_{\tau}\tau + \bm{\lambda}^{\top}(\bm{w} - \tau\bm{1}_N^{\top}).
\ee
The dual variable constraints and the first order optimality conditions of \eqref{quad} can be written as 
\begin{eqnarray}
	GG^{\top}\bm{w} - G\bm{g} + \bm{\lambda} & = & \bm{0}, \label{grad} \\
	\bm{1}^{\top}\bm{\lambda} + \lambda_{\tau}& = & 0,  \label{eq_lambda} \\
	\bm{\lambda}(j)  \geq  0,\quad \bm{w}(j) & \leq &\tau, \quad j \not\in\Omega \label{ineq_con} \\
	\bm{\lambda}(j)(\bm{w}(j) - \tau) & = & 0, \quad j \not\in\Omega  \label{slack_ineq} \\
	\bm{w}(j) & = &\tau, \quad j \in\Omega \label{eq_con} \\
	\lambda_{\tau} \geq  0, \quad \lambda_{\tau}\tau & = &  0. \label{slack_tau}
\end{eqnarray}
\begin{flalign}
	\text{From \eqref{grad}}\quad && \bm{w} = \underbrace{(GG^{\top})^{-1}G}_{E}\bm{g} -  \underbrace{(GG^{\top})^{-1}}_{F}\bm{\lambda}, &&
	\label{opt-w}
\end{flalign}
where the matrices $E,F$ can be computed explicitly: $[E]_{ij} = 1_{\{j\geq i\}}$ and $[F]_{ij} = N - \max\{i,j\}+1$, and using \eqref{eq_lambda} gives 
\be \begin{split}
	\hskip-0.3cm \bm{w}(j) = \sum_{i=j}^N\bm{g}(i) + (N-j+1)\lambda_{\tau} + \sum_{i=j+1}^N(i-j)\bm{\lambda}(i),
	\label{opt-w-spr}
\end{split} \ee 
for $j\in [N]$. Now suppose let $i_k$ the last change point and suppose that 
$M = \max_{j=i_k,\ldots,N}\left\{ \sum_{n=j}^{N}\bm{g}(n)\right\}.$
Consider first the case where $M<0$ and suppose that $\tau > 0$. In this case from \eqref{slack_tau} we have $\lambda_{\tau} = 0$. Plugging this into \eqref{opt-w-spr} for $j=N$ we get $\bm{w}(N) = \bm{g}(N) < 0 \Rightarrow \bm{w}(N) - \tau < 0 \stackrel{\eqref{slack_ineq}}{\Rightarrow} \bm{\lambda}(N) = 0$. Decreasing $j$ and proceeding similarly we get $\bm{\lambda}(N) = \bm{\lambda}(N-1) = \ldots = \bm{\lambda}(i_k+1) = 0$. Now for $j=i_k$ we get
$
	\bm{w}(i_k) = \sum_{j=i_k}^N\bm{g}(j) < 0,
$
and \eqref{eq_con} cannot be satisfied for $\tau > 0$. Therefore $\tau = 0$. 

Now assume that $M>0$, and that this maximum occurs at the location $N-l$. 
Then by plugging $j=N-l$ into \eqref{opt-w-spr} and the nonnegativity of the dual variables we have that $\bm{w}(N-l) \geq M \Rightarrow \tau \geq M \Rightarrow \lambda_{\tau} = 0$. We proceed as before: For $j=N$ \eqref{opt-w-spr} gives $\bm{w}(N) < M \leq \tau \Rightarrow \bm{\lambda}(N) = 0$. And similarly $\bm{\lambda}(N) = \bm{\lambda}(N-1) = \ldots = \bm{\lambda}(N-l+1) = 0$. Plugging this into \eqref{opt-w-spr} for $j=N-l$ we get $\bm{w}(N-l) = M$. Since $\bm{w}(N-l) \leq \tau$ we get that $\tau = M$. \qed

\noindent Lemma \ref{opt_tau} allows us to estimate the regularizer that minimizes \eqref{minmax} by estimating $\tau^{avg}$ that arises in \eqref{sd_sub}, and consequently set the regularizer $\lambda = \tau^{avg}\sigma$. In general $\tau^{avg} = M(N-i_k+1)$ where $M(n)$ is the expected value of the maximum of a standard Gaussian random walk of $n$ steps, truncated at 0. $M(1) = 1/\sqrt{2\pi}$, and $M(2) = (1+\sqrt{2})/(2\sqrt{\pi})$. In general $M(n)$ cannot be computed explicitly, 
but can be easily upper bounded:
Let $X_i\sim \mathcal{N}(0,1)$, $S_k = \sum_{i=1}^kX_i$, and $E_n = \max_{1\leq k \leq n}S_k$.  Using the L\'{e}vy inequality
\[ \begin{split}
	\mathbb{P}\left(E_n \geq x\right) &\leq 2\mathbb{P}(S_n \geq x) \Rightarrow  \\
	M(n)  = \int_{0}^{\infty}\mathbb{P}\left(E_n \geq x\right) \mathrm{d} x 
		 & \leq \int_{0}^{\infty}\hskip-0.2cm\mathrm{erfc}\left(\frac{x}{\sqrt{2n}}\right) \mathrm{d}x = \sqrt{\frac{2n}{\pi}}.
\end{split} \]



\section{Discussion}

The \eqref{PDN} problem for monotone signals was first discussed in \cite{DJM13} in the context of monotone regression without regularization. There an upper bound was derived and the relation of the minimax error with the PTC for the \eqref{PCS} problem was established. For the \ref{PCS} problem \cite{PP13a} examined, in the context of sparse deconvolution, the case of signals where $\bm{x}(i+1) - \gamma\bm{x}(i)$ is sparse and non-negative, with $0<\gamma<1$ and close to 1, and identified the best, average, and worst cases depending on the location of the change points, without deriving a closed form expression. To the best of our knowledge, this paper presents for the first time an non-asymptotic closed form expression that captures the dependence on both the number and the location of the change points, and also characterizes the optimal regularizer.
\begin{figure}[t]
	\centering
	\includegraphics[trim={10mm 4mm 5mm 7mm}, clip,width=0.75\textwidth]{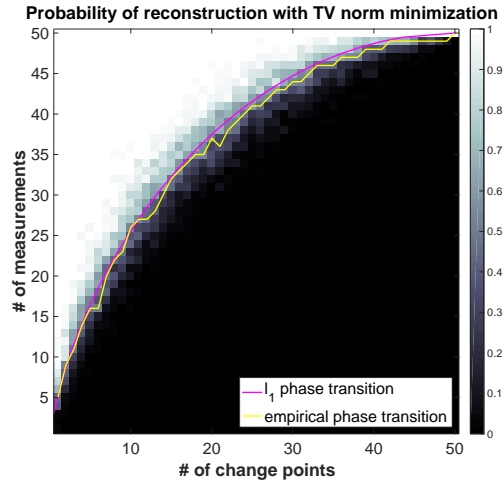}
	\caption{Empirical calculation of reconstruction probability for sparsely varying signals. 50-dimensional piecewise constant signals were constructed with variable number of change points $k$ and locations chosen uniformly at random. For each signal a random Gaussian sensing matrix was constructed with variable number of rows (measurements) $m$. Reconstruction was attempted by minimizing the TV norm subject to the measurements, and for each pair $(k,m)$, 50 iterations were performed. The probability of success (color coded in the background) undergoes a phase transition. The empirical $50\%$ success line (yellow) lies very close to the PTC for sparse signals (magenta) as is theoretically computed in \cite{DMM09}.
	}
	\label{PTC-TV}
\end{figure}
\noindent Future work includes the case of non-monotone sparsely varying signals, with the TV norm acting as the structure inducing function. The striking resemblance between the average SD \eqref{ptc-avg} and the PTC for the case of non-negative sparse signals \citep{DMM09}, motivates a comparison between the average SD for this case and the PTC for recovering sparse signals using the $l_1$ norm. While a closed form solution for the SD is not available, some upper bounds appear in \cite{CX15}, simulations suggest a close match (Fig.~\ref{PTC-TV}).

\section{Acknowledgements}
Part of the work was performed while the author was with the Department of Statistics, Columbia University, NewYork, NY 10027. The author thanks L. Paninski, M. McCoy and J. Tropp for useful discussions.

\vfill\pagebreak

\bibliographystyle{chicago}
\bibliography{../mybib}

\begin{thebibliography}{}

\bibitem[\protect\citeauthoryear{Amelunxen, Lotz, McCoy, and Tropp}{Amelunxen
  et~al.}{2014}]{ALMT14}
Amelunxen, D., M.~Lotz, M.~B. McCoy, and J.~A. Tropp (2014).
\newblock Living on the edge: Phase transitions in convex programs with random
  data.
\newblock {\em Information and Inference\/}, iau005.

\bibitem[\protect\citeauthoryear{Cai and Xu}{Cai and Xu}{2015}]{CX15}
Cai, J.-F. and W.~Xu (2015).
\newblock Guarantees of total variation minimization for signal recovery.
\newblock {\em Information and Inference\/}, iav009.

\bibitem[\protect\citeauthoryear{Donoho, Johnstone, and Montanari}{Donoho
  et~al.}{2013}]{DJM13}
Donoho, D., I.~Johnstone, and A.~Montanari (2013).
\newblock Accurate prediction of phase transitions in compressed sensing via a
  connection to minimax denoising.
\newblock {\em IEEE Trans. Informat. Theory\/}~{\em 59\/}(6), 3396--3433.

\bibitem[\protect\citeauthoryear{Donoho, Maleki, and Montanari}{Donoho
  et~al.}{2009}]{DMM09}
Donoho, D., A.~Maleki, and A.~Montanari (2009).
\newblock Message-passing algorithms for compressed sensing.
\newblock {\em Proceedings of the National Academy of Sciences\/}~{\em
  106\/}(45), 18914.

\bibitem[\protect\citeauthoryear{Oymak and Hassibi}{Oymak and
  Hassibi}{2012}]{oymak2012relation}
Oymak, S. and B.~Hassibi (2012).
\newblock On a relation between the minimax risk and the phase transitions of
  compressed recovery.
\newblock In {\em Communication, Control, and Computing (Allerton), 2012 50th
  Annual Allerton Conference on}, pp.\  1018--1025. IEEE.

\bibitem[\protect\citeauthoryear{Pnevmatikakis and Paninski}{Pnevmatikakis and
  Paninski}{2013}]{PP13a}
Pnevmatikakis, E. and L.~Paninski (2013).
\newblock Sparse nonnegative deconvolution for compressive calcium imaging:
  algorithms and phase transitions.
\newblock In {\em Advances in Neural Information Processing Systems},
  Volume~26, pp.\  1250--1258.

\end{thebibliography}

\end{document}